\documentclass[12pt]{amsart}
\usepackage{amsmath, amsfonts,amsthm,amssymb,amscd, verbatim,mathdots, graphicx,color,multirow,booktabs,tikz,adjustbox, enumitem, tikz-cd}
\usepackage{chngcntr}
\counterwithin{table}{section}
\def\classification#1{\def\@class{#1}}
\classification{\null}
\usepackage[margin=1.5cm]{geometry}
\usepackage{url}

\renewcommand{\wr}{\mathop{\mathrm{wr}}}

\DeclareFontFamily{OT1}{rsfs}{}
\DeclareFontShape{OT1}{rsfs}{n}{it}{<-> rsfs10}{}
\DeclareMathAlphabet{\mathscr}{OT1}{rsfs}{n}{it}

\newcommand{\Fq}{\mathbb{F}_q}

\newcommand{\bF}{\mathbb{F}}

\newcommand{\SL}{\mathrm{SL}}

\newcommand{\Sp}{\mathrm{Sp}}

\newcommand{\PSL}{\mathrm{PSL}}

\newcommand{\PGammaL}{\mathrm{P\Gamma L}}
\newcommand{\GL}{\mathrm{GL}}

\newcommand{\base}{\mathrm{b}}
\newcommand{\Base}{\mathrm{B}}
\newcommand{\Height}{\mathrm{H}}
\newcommand{\Irred}{\mathrm{I}}
\def\ov{\overline}

\def\a{\alpha}
\newcommand{\RC}{\mathrm{RC}}

\newtheorem{prop}{Proposition}[section]
\newtheorem{thm}{Theorem}
\newtheorem{conj}[prop]{Conjecture}

\newtheorem{example}[prop]{Example}
\newtheorem{cor}[thm]{Corollary}
\newtheorem{lem}[prop]{Lemma}
\newtheorem{defn}[prop]{Definition}
\theoremstyle{definition}
\newtheorem{remark}[prop]{Remark}
\numberwithin{equation}{section}
\def\o{\omega}

\theoremstyle{definition}

\begin{document}
\title{Irredundant bases for finite groups of Lie type}  

\begin{abstract}
 We prove that the maximum length of an irredundant base for a primitive action of a finite simple group of Lie type is bounded above by a function which is a polynomial in the rank of the group. We give examples to show that this type of upper bound is best possible.
\end{abstract}
\author{Nick Gill}
\address{School of Mathematics and Statistics, The Open University, Walton Hall, Milton Keynes, MK7 6AA, UK}
\email{nick.gill@open.ac.uk}

\author{Martin W. Liebeck}
\address{Department of Mathematics, Imperial College London, London, SW7 2AZ, UK}
\email{m.liebeck@imperial.ac.uk }

\maketitle

\section{Introduction}

\subsection{Main results} 
Let $G$ be a group acting on a set $\Omega$. Let $\ell$ be a non-negative integer and let $\Lambda = [\omega_1,\dots,\omega_\ell]$ be a sequence of points $\omega_1,\dots, \omega_\ell$ drawn from $\Omega$; we write $G_{(\Lambda)}$ or $G_{\omega_1, \omega_2, \dots, \omega_\ell}$  for the pointwise stabilizer. If $\ell=0$, so $\Lambda$ is empty, then we set $G_{(\Lambda)}=G$.

The sequence $\Lambda$ is called a \emph{base} if $G_{(\Lambda)}=\{1\}$; the sequence $\Lambda$ is called \emph{irredundant} if
\[
  G_{\omega_1,\dots, \omega_{k-1}} \gneq G_{\omega_1,\dots, \omega_{k-1}, \omega_k}
\]
for all $k=1,\dots, \ell$. The size of the longest possible irredundant base is denoted $\Irred(G,\Omega)$. 

The main result of this paper shows that for any primitive action of a simple group of Lie type, the size of an irredundant base is bounded by a polynomial function of the rank of the group.

\begin{thm}\label{t: i}
If $G$ is a simple group of Lie type of rank $r$ acting primitively on a set $\Omega$, then $\Irred(G,\Omega) \leq Cr^8$, where $C$ is an absolute constant. This holds with $C = 174$.
\end{thm}

The degree 8 of the polynomial bound is probably far from sharp -- but as discussed in Section \ref{disc}, there are examples showing that this degree must be at least 2. Also there is no general complementary lower bound for $\Irred(G,\Omega)$ that grows with $r$, as shown by Example \ref{complem} at the end of in Section \ref{s: proof irred}.

An upper bound on $\Irred(G,\Omega)$ implies an upper bound on a host of other statistics associated with the action of $G$ on $\Omega$. Consider, again, the sequence $\Lambda$, defined above. We call $\Lambda$ a \emph{minimal base} if it is a base and, furthermore, no proper subsequence of $\Lambda$ is a base. We denote the minimum size of a minimal base $\base(G,\Omega)$, and the maximum size of a minimal base $\Base(G, \Omega)$.

We say that $\Lambda$ is \emph{independent} if, for all $k=1,\dots, \ell$, we have $G_{(\Lambda)}\neq G_{(\Lambda\setminus \o_k)}$.
 We define the \emph{height} of $G$ to be the maximum size of an independent sequence, and we denote this quantity $\Height(G, \Omega)$. 


The last statistic of interest to us is the \emph{relational complexity} of the action of $G$ on $\Omega$, denoted $\RC(G,\Omega)$. The definition of this is slightly involved and can be found in \cite{cherlin_martin} where it is given the name ``\emph{arity}''.
 
It is easy to verify the following inequalities \cite{glodas}:
\begin{equation}\label{eq: in1}
 \base(G,\Omega) \leq \Base(G,\Omega) \leq \Height(G,\Omega) \leq \Irred(G,\Omega).
\end{equation} 
Less obvious, but still rather elementary is the following \cite{glodas}:
\begin{equation}\label{eq: in2}
\RC(G,\Omega)\leq \Height(G,\Omega)+1.
\end{equation}

Theorem~\ref{t: i} and inequalities \eqref{eq: in1} and \eqref{eq: in2} immediately yield the following corollary.

\begin{cor}\label{c: i}  If $G$ is simple of Lie type of rank $r$ acting primitively on a set $\Omega$, then each of 
$\base(G,\Omega)$, $\Base(G,\Omega)$, $\Height(G,\Omega)$ and $\Irred(G,\Omega)$ is at most $Cr^8$ while $\RC(G,\Omega)$ is less than $Cr^8+1$, where $C$ is as in Theorem $\ref{t: i}$.
\end{cor}

We can also deduce an upper bound for primitive actions of {\it almost} simple groups:

\begin{cor}\label{c: ii}
Let $G$ be an almost simple group, with socle a simple group of Lie type of rank $r$ over $\Fq$, where $q=p^f$ ($p$ prime). If $G$ acts primitively on a set $\Omega$, then 
\[
 \Irred(G,\Omega) \le 177r^8+\pi(f),
\]
where $\pi(f)$ is the number of primes, counted with multiplicity, dividing the integer $f$. 
\end{cor}

Example \ref{pif} in Section \ref{corpf} shows that the term $\pi(f)$ in the upper bound cannot be avoided.

Our main tool for proving Theorem \ref{t: i} is the following result on maximal subgroups of finite groups of Lie type. In the statement, we let $G(q)=(\overline G^F)'$ be a simple group of Lie type over $\Fq$, where $\overline{G}$
is the corresponding simple adjoint algebraic group over $\overline{\Fq}$ and $F$ is a Frobenius endomorphism. Let $p$ be the characteristic of $\Fq$. For a rational representation $\rho:\overline G \mapsto GL_n(\overline{\Fq})$, and a closed subgroup $\overline H$ of $\overline G$, we define $\hbox{deg}_{\rho}(\ov H)$ to be the degree of the image $\rho(\ov H)$ as a subvariety of $GL_n(\ov \Fq)$. We give some basic  definitions and results about degree in Section \ref{irrsec}.

\begin{thm}\label{t: alg}
Let $G(q)= (\overline G^F)'$ be a finite simple group of Lie type as above, and let $G$ be an almost simple group with socle $G(q)$. Let $M$ be a maximal subgroup of $G$, and set $M_0 = M \cap G(q)$. Let $d = \dim \ov G$. Then
one of the following holds:
\begin{itemize} 
\item[$(1)$] $M_0 = \overline{M}^F\cap G(q)$, where $\overline{M}$
is a closed $F$-stable subgroup of $\overline{G}$ of
positive dimension; moreover, 
\begin{itemize}
\item[{\rm (a)}] $|\ov M:\ov M^0| \le |W(\ov G)|$, the order of the Weyl group of $\ov G$, and
\item[{\rm (b)}] excluding the cases where $(\ov G,\ov M,p) = (C_r,D_r,2)$ or $(C_3,G_2,2)$, if we let ${\rm ad}:\ov G \mapsto GL(L(\ov G))$ be the adjoint representation, then
\[
\hbox{deg}_{ad}(\ov M) \le |W(\ov G)|\hbox{deg}_{ad}(\ov G) \le |W(\ov G)|\,2^{d^2}.
\]
\end{itemize}

\item[$(2)$] $M_0 = G(q_0)$, a subgroup of the same type as $G$ (possibly twisted) over a
subfield $\mathbb{F}_{q_0}$ of $\Fq$.

\item[$(3)$] $|M_0|\leq 2^{d^2}$.
\end{itemize}
\end{thm}

\subsection{Context for, and possible improvements to, Theorem~\ref{t: i}} \label{disc}

We think of Theorem~\ref{t: i} as being a version of the Cameron--Kantor conjecture for irredundant bases. The Cameron--Kantor conjecture, which was stated in \cite{cameron-kantor, cameron} and proved in \cite{lieshal2}, asserts the existence of an absolute upper bound for the size of a minimal base for the non-standard actions of the almost simple groups. (A \emph{standard action} of an almost simple group $G$ with socle $S$ is a transitive action where either $S=A_n$ and the action is on subsets or uniform partitions of $\{1,\dots, n\}$, or $G$ is classical and the action is a subspace action.)

In \S\ref{s: irred ck} we explain exactly how Theorem~\ref{t: i} is connected to the Cameron--Kantor conjecture and we give a number of examples that clarify why Theorem~\ref{t: i} is, in a certain sense, the best possible ``Cameron--Kantor like statement'' that can be made for irredundant bases. In particular, we give examples to show that 
\begin{itemize}
\item[(i)] even for non-standard actions, the bound $Cr^8$ in Theorem \ref{t: i} really needs to depend on $r$ and is not absolute;
\item[(ii)] Theorem \ref{t: i} only holds for primitive actions of \emph{simple} groups of Lie type -- it does not extend to actions of almost simple groups in general (although we do prove Corollary \ref{c: ii} for these);
\item[(iii)] likewise, Theorem \ref{t: i} does not extend to transitive actions of simple groups of Lie type in general.
\end{itemize}

Although item (i) implies that the upper bound given in Theorem~\ref{t: i} is necessarily a function of $r$, it is undoubtedly true that the particular function of $r$ we have given -- $174r^8$ -- can be improved. A construction of Freedman, Kelsey and Roney-Dougal (personal communication) implies that any polynomial upper bound must have degree at least $2$; our guess is that an upper bound which is quadratic in $r$ may hold in general. 

A heuristic supporting this guess follows from the fact that $\Irred(G,\Omega)\leq \ell(G)$, where $\ell(G)$ is the maximum length of a subgroup chain in the simple group of Lie type $G$. Writing $p$ for the field characteristic, $U$ for a Sylow $p$-subgroup of $G$, and $\Phi^+$ for the associated set of positive roots, we know that there exist constants $c_1, c_2$ such that
\[
 c_1r^2\log_pq \leq |\Phi^+|\log_pq =\ell(U) < \ell(G) < \log_2|G| \leq c_2r^2\log_2 q.
\]
More information about $\ell(G)$ can be found in \cite{solomon_turull}.

Theorem~\ref{t: i} is the second recent success in trying to extend well-known results about bases to statements about irredundant bases; the first was achieved by Kelsey and Roney-Dougal \cite{kelsey2021relational} extending a result of Liebeck \cite{liebeck_base}. 

\subsection{Proofs and the structure of the paper}

In \S\ref{irrsec} we present a number of definitions and results pertaining to the degree of an affine variety; these include, in particular, a statement of (one version of) B\'ezout's theorem on the degree of the intersection of a number of algebraic varieties.

In \S\ref{algsec} we prove Theorem~\ref{t: alg}. The proof uses various results from the literature on the subgroup structure of algebraic groups \cite{LSei-Invent, LS04}. 

In \S\ref{s: proof irred} we prove Theorem~\ref{t: i}; the proof makes use of both Theorem~\ref{t: alg} and B\'ezout's theorem. Corollary \ref{c: ii} is deduced in \S\ref{corpf}.

The comparison of Theorem \ref{t: i} with the Cameron-Kantor conjecture, and the relevant examples mentioned above, are given in \S\ref{s: irred ck}, which is the final section of the paper.

\subsection{Acknowledgments}

The authors thank Scott Harper, Scott Hudson and Chris Wuthrich for a number of helpful conversations. 


\section{Degree of an affine variety}\label{irrsec}

Our proof of Theorem \ref{t: i} is carried out by combining Theorem \ref{t: alg} with B\'ezout's theorem on the degree of the intersection of a number of algebraic varieties. We need a version of B\'ezout's theorem that holds for affine varieties and is due to Heintz \cite{heintz}.

In what follows we consider subsets of some affine space, $\mathbb{A}^n$, over an algebraically closed field $k$. A set $X$ in $\mathbb{A}^n$ is called \emph{locally closed} if $X=V\cap W$, where $V$ is open and $W$ is closed (in the Zariski topology). A set $X$ is called \emph{constructible} if it is a finite disjoint union of locally closed sets. Note that the intersection of a finite number of constructible sets is constructible. Note too that any variety in $\mathbb{A}^n$ is constructible. From here on $X$ is a constructible set.

\begin{defn}\cite[Definition~1 and Remark~2]{heintz}\label{d: heintz}
 If $X$ is an irreducible variety of dimension $r$ in $\mathbb{A}^n$, then the \emph{degree} of $X$, written $\deg(X)$, is defined to be
 \[
  \sup\{|E\cap X| \, \mid \, E \textrm{ is an $(n-r)$-dimensional affine subspace of $\mathbb{A}^n$ such that $E\cap V$ is finite}\}.
 \]

If $X$ is a constructible set and $\mathcal{C}$ is the set of irreducible components of the closure of $X$, then we define
\begin{equation}\label{degsum}
  \deg(X) = \sum\limits_{C\in \mathcal{C}} \deg(C).
\end{equation}
\end{defn}

Note that if $X$ is an irreducible variety of dimension $0$, then we have $\deg(X)=1$. Thus, if $X$ is any variety of dimension $0$, irreducible or not, $\deg(X)=|X|$.

Now the main result that we need concerning degree is the following version of B\'ezout's Theorem.

\begin{prop}\label{t: bezout}\cite[Theorem~1]{heintz}
Let $X$ and $Y$ be constructible sets in $\mathbb{A}^n$.
Then \[\deg(X\cap Y) \leq \deg(X) \cdot \deg(Y).\]
\end{prop}

This proposition obviously generalizes to the intersection of more than two varieties: If $X_1,X_2,\dots, X_k$ are constructible sets in $\mathbb{A}^n$, then
\[
 \deg(X_1\cap X_2\cap \cdots \cap X_k)\leq \deg(X_1)\cdot\deg(X_2)\cdots\deg(X_k).
\]
(We are implicitly using the fact that the intersection of two constructible sets is constructible.) 

A useful corollary of Proposition~\ref{t: bezout} is the following fact connecting the degree of an affine variety to the degree of its defining polynomials. We make use of the fact, noted by Heintz \cite[p.247]{heintz}, that the degree of a hypersurface in $\mathbb{A}^n$ is equal to the degree of its defining polynomial.

\begin{lem}\label{prod}
 Suppose that an affine variety $X$ in $\mathbb{A}^n$ is defined by polynomials $f_1,\dots, f_r$ of degree at most $e$. Then
 \[
\deg(X)\leq e^r.  
 \]
\end{lem}

\begin{proof}
By definition $X=V(f_1,\dots, f_r) = \bigcap\limits_{i=1}^r V(f_i)$ where, for $i=1,\dots, r$,  $V(f_i)$ is the hypersurface defined by the polynomial $f_i$. We remarked that $\deg(V(f_i))=\deg(f_i)$, hence Proposition \ref{t: bezout} implies that
\[
\deg(X)\leq \deg(V(f_1))\cdots \deg(V(f_r))=\deg(f_1)\cdots \deg(f_r) \leq e^r.
\]
\end{proof}

As mentioned in the Introduction, if $\overline{G}$ is an affine algebraic group over an algebraically closed field $k$, then for a rational representation $\rho:\overline G \mapsto GL_n(k)$, and a closed subgroup $\overline H$ of $\overline G$, we define $\hbox{deg}_{\rho}(\ov H)$ to be the degree of the image $\rho(\ov H)$ as a subvariety of $GL_n(k)$. 
From (\ref{degsum}), we have 
\begin{equation}\label{g0deg}
\deg_{\rho}(\overline{H}) = |\overline{H}: \overline{H}^0|\deg_{\rho}(\overline{H}^0) \geq \deg_{\rho}(\overline{H}^0).
\end{equation}

\section{Proof of Theorem \ref{t: alg}}\label{algsec}

As in Theorem \ref{t: alg}, let $G(q) = (\ov G^F)'$ be a simple group of Lie type over $\Fq$, where $\ov G$ is a simple algebraic group over $K=\ov \Fq$, and let $G$ be an almost simple group with socle $G(q)$. Let $M$ be a maximal subgroup of $G$, and set $M_0 = M\cap G(q)$. Let $d = \dim \ov G$ and let $p$ be the characteristic of $\Fq$.

Suppose first that $G(q)$ is a classical group, so that $\ov G$ is the corresponding classical algebraic group. Let $V$ be the natural module for $\ov G$, and let $n = \dim V$. We shall apply \cite[Thms. $1'$ and 2]{LSei-Invent}. We postpone consideration of the cases where $G(q) = PSL_n(q)$, $Sp_4(2^e)$ or $P\Omega_8^+(q)$ and the group $G$ contains an element in the coset of a graph automorphism (a triality graph automorphism in the last case). Assuming that these cases do not pertain, in \cite{LSei-Invent}, six classes 
${\mathcal C}_i$ of closed subgroups of $\ov G$ are defined, and it is proved that one of the following holds:
\begin{itemize}
\item[(i)] $M_0 = \overline M^F\cap G(q)$ for some $F$-stable member $\ov M \in {\mathcal C}:=\bigcup_1^6 {\mathcal C}_i$,
\item[(ii)] $M_0 = G(q_0)$, a subgroup of the same type as $G(q)$ (possibly twisted) over a
subfield $\mathbb{F}_{q_0}$ of $\Fq$,
\item[(iii)] $M_0$ is almost simple, and $F^*(M_0)$ is irreducible on $V$ (and not of the same type as $G(q)$).
\end{itemize}
In case (ii), conclusion (2) of Theorem \ref{t: alg} holds.

Consider now case (i). The only finite members of ${\mathcal C}$ are 
\begin{itemize}
\item subgroups of type $O_1(K) \wr S_n = 2^n.S_n$ in $O_n(K)$ with $p\ne 2$ (these lie in the class ${\mathcal C}_2$), and 
\item extraspecial-type subgroups $r^{2m}.Sp_{2m}(r)$ ($r$ prime, $n=r^m$) or $2^{2m}.O^{\pm}_{2m}(2)$ ($n=2^m$) (these lie in the class ${\mathcal C}_5$).
\end{itemize}
A simple check shows that these subgroups have order less than $2^{d^2}$, as required for conclusion (3) of Theorem \ref{t: alg}.

All the other members of ${\mathcal C}$ are infinite, in which case 
\begin{equation}\label{conc1}
M_0 = \overline{M}^F\cap G(q), \hbox{ where } \overline{M} \hbox{ is a maximal closed $F$-stable subgroup of $\overline{G}$ of positive dimension,}
\end{equation}
as in (1) of Theorem \ref{t: alg}. 

Now consider case (iii) above. If $F^*(M) \not \in {\rm Lie}(p)$, then an unpublished manuscript of Weisfeiler \cite{Weis}, subsequently improved and developed in \cite{Collins}, shows that $|M| < n^4(n+2)!$, which is less than $2^{d^2}$, as in (3) of Theorem \ref{t: alg}.
And if $F^*(M)  \in {\rm Lie}(p)$, then \cite[Thm. 1]{SeiTest} shows that (\ref{conc1}) holds. 

To complete the proof of Theorem \ref{t: alg} in the case where $G$ is classical (apart from the postponed cases), it remains to prove the bounds for $|\ov M:\ov M^0|$, $\deg_{ad}(\ov M)$ and $\deg_{ad}(\ov G)$ for $\ov M$ in (1) of Theorem \ref{t: alg}. The bound 
$|\ov M:\ov M^0| \le |W(\ov G)|$ follows by simply inspecting the structure of the members of ${\mathcal C}$; equality occurs when $\ov M = N_{\ov G}(T)$, where $T$ is a maximal torus (these subgroups are in class ${\mathcal C}_2$ for $SL(V)$ and $SO(V)$). 

To establish the degree bounds, we first prove

\vspace{2mm}
\noindent {\bf Claim } Let $M_0 = \overline{M}^F\cap G(q)$ be as in (\ref{conc1}). Then with two exceptions, $\ov M^0$ acts reducibly on some $\ov G$-composition factor of the adjoint module $L(\ov G)$. The two exceptions are $(\ov G, \ov M, p) = (Sp_n,SO_n,2)$ or $(Sp_6,G_2,2)$.

\vspace{2mm}
{\it Proof of Claim} The composition factors of $L(\ov G)$ are given in \cite[Prop. 1.10]{LSeiTAMS}. Also $L(\ov M) \subseteq L(\ov G)$. First consider $M_0 = \ov M^F\cap G(q)$ as in (i). Inspecting $\ov M^0$ for  $\ov M \in {\mathcal C}$, we see that $L(\ov M)$ maps to a proper subspace of some composition factor of $L(\ov G)$, with the exception of $(\ov G, \ov M, p) = (Sp_n,SO_n,2)$, proving the claim for $M_0$ as in (i). Finally, for $M_0$ as in (iii), the group $\ov M^0$ is simple, and \cite[Thm. 4]{LSeiTAMS} shows that the only case where $L(\ov M)$ does not map to a proper subspace of some composition factor of $L(\ov G)$ is $(\ov G, \ov M, p) =(Sp_6,G_2,2)$. This completes the proof of the Claim. 

\vspace{2mm} We now use the Claim to deduce the required degree bounds. Let $M, \ov M$ be as in (\ref{conc1}), and exclude the exceptions in the Claim, so that $\ov M^0$ acts reducibly on some composition factor of $L(\ov G)$. If also $\ov M$ is reducible, then there is a subspace $W$ of $L(\ov G)$ such that 
\[
\ov M = {\rm stab}_{\ov G}(W).
\]
This defines $\ov M$ by the polynomials defining $\ov G$ in the adjoint representation, together with some linear equations, and hence by Lemma~\ref{prod}, we have 
\[
\deg_{ad}(\ov M) \le \deg_{ad}(\ov G).
\]
On the other hand, if $\ov M$ acts irreducibly on every composition factor of $L(\ov G)$, then by the Claim, there is a composition factor $V$ such that $V\downarrow \ov M^0 = \bigoplus_1^t V_i$, where each $V_i$ is irreducible for $\ov M^0$ and $t\ge 2$. Set
\[
\ov M^1 = \bigcap_1^t {\rm stab}(V_i),
\]
so that $\ov M^0 \le \ov M^1 \triangleleft \ov M$. As above we see that $\deg_{ad}(\ov M^1) \le \deg_{ad}(\ov G)$, and so by the remarks after Lemma \ref{prod}, we have $\deg_{ad}(\ov M) \le |\ov M:\ov M^1|\,\deg_{ad}(\ov G)$. We have seen that 
$|\ov M:\ov M^0| \le |W(\ov G)|$, so it follows that 
\[
\deg_{ad}(\ov M) \le |W(\ov G)|\,\deg_{ad}(\ov G),
\]
as required for (1) of Theorem \ref{t: alg}. Finally, in the adjoint representation, $\ov G$ is defined by $d^2$ quadratic polynomials expressing preservation of the Lie bracket on $L(\ov G)$, so $\deg_{ad}(\ov G) \le 2^{d^2}$. Note that the exceptional cases $(Sp_n,SO_n,2)$, $(Sp_6,G_2,2)$ in the Claim are also excepted in part (i)(b) of Theorem \ref{t: alg}. Hence the proof of the theorem for $G$ classical is now complete, apart from the postponed cases where $G(q) = PSL_n(q)$, $Sp_4(2^e)$ or $P\Omega_8^+(q)$ and $G$ contains an element in the coset of a graph automorphism. 

Now consider the excluded cases. Suppose first that $G(q) = PSL_n(q)$. In this case, the collection ${\mathcal C}$ is extended in \cite{LSei-Invent} to a collection ${\mathcal C}'$, and it is proved that conclusion (i), (ii) or (iii) above holds, with ${\mathcal C}'$ replacing ${\mathcal C}$. The only subgroups in ${\mathcal C}'\setminus {\mathcal C}$ are stabilizers of pairs $\{U,W\}$ of subspaces of $V$ such that either $U\subseteq W$ or $V = U\oplus W$. The above proof shows that these subgroups satisfy (1) of Theorem \ref{t: alg}. In the other cases, where $G(q) = Sp_4(2^e)$ or $P\Omega_8^+(q)$, the maximal subgroups of $G$ are listed in \cite[Tables 8.14, 8.50]{BHRD}. Inspection of these lists shows that (1), (2) or (3) of Theorem \ref{t: alg} holds (using the same argument as above to bound the degree of $\ov M$). This completes the proof of Theorem \ref{t: alg} for $G(q)$ a classical group.

Suppose finally that $G(q)$ is an exceptional group of Lie type. The proof runs along similar lines. First we use \cite[Thm. 8]{LSeidur}, which gives the possibilities for the maximal subgroup $M$. These are:
\begin{itemize}
\item[(i)] $M_0= \ov M^F\cap G(q)$, where $\ov M$ is a maximal closed $F$-stable subgroup of $\ov G$ of positive dimension;
\item[(ii)] $M_0= G(q_0)$, a subgroup of the same type as $G$ (possibly twisted) over a
subfield $\mathbb{F}_{q_0}$ of $\Fq$,
\item[(iii)] $M_0$ is an ``exotic local" subgroup $3^3.SL_3(3)<F_4$, $3^{3+3}.SL_3(3) < E_6$, $5^3.SL_3(5) < E_8$ or $2^{5+10}.SL_5(2) < E_8$;
\item[(iv)] $M_0$ is the ``Borovik subgroup " $(Alt_5\times Alt_6).2^2 < E_8$;
\item[(v)] $M_0$ is almost simple with socle $M_1$, and one of the following holds:
\begin{itemize}
\item[(a)] $M_1 \not \in {\rm Lie}(p)$: the possibilities for $M_0$ are listed in \cite[Thm. 4]{LSeidur};
\item[(b)] $M_1 = M(q_1) \in {\rm Lie}(p)$, ${\rm rank}(M_1) \le \frac{1}{2}{\rm rank}(\ov G)$, and one of:
\begin{itemize}
\item $q_1\le 9$
\item $M_1 = A_2^{\pm}(16)$
\item $M_1$ has rank 1 and $q_1 \le (2,p-1)\cdot t(\ov G)$, where $t(\ov G) = 12,68,124,388,1312$ according as $\ov G = G_2,F_4,E_6,E_7,E_8$, respectively.
\end{itemize}
\end{itemize}
\end{itemize}
In cases (iii), (iv) and (v) we check that $|M_0| < 2^{d^2}$, as in (3) of Theorem \ref{t: alg}; and case (ii) is (2) of the theorem. Finally, in case (i), the list of possibilities for $\ov M$ is given in \cite[Thm. 8]{LSeidur} and the references quoted there. We can check that $|\ov M:\ov M^0| \le |W(\ov G)|$, and also that $\ov M^0$ acts reducibly on some $\ov G$-composition factor of $L(\ov G)$ (see also \cite{LSeiadj} for this). Now we can argue exactly as in the classical case to obtain the required bounds on $\deg_{ad}(\ov G)$ for $\ov M$ for (1) of Theorem \ref{t: alg}.

This completes the proof of Theorem \ref{t: alg}.

\section{Proof of Theorem~\ref{t: i}}\label{s: proof irred}

Let $G$ be a simple group of Lie type of rank $r$ over $\mathbb{F}_q$ with $G=(\overline{G}^F)'$, where $\overline{G}$ is the corresponding simple algebraic group over $\overline{\mathbb{F}_q}$ and $F$ is a Frobenius endomorphism. Let $d = \dim \ov G$ and $p = {\rm char}(\Fq)$.

We write $G_1$ for a maximal subgroup of $G$. We consider the action of $G$ on $\Omega$, the set of cosets of $G_1$. We suppose that we have a stabilizer chain,
\begin{equation}\label{e: stab chain}
 G>G_1>G_2 > \cdots > G_k=\{1\}
\end{equation}
where $G_i=G_{i-1}\cap G_1^{g_i}$ for some $g_i\in G$ ($i=1,\ldots,k$).

Theorem~\ref{t: alg} gives three possibilities for $G_1$. 

\subsection{Case 1 of Theorem \ref{t: alg}}\label{s: c1} In this case we have $G_1=\overline{G_1}^F\cap G$ where $\overline{G_1}$ is a closed $F$-stable subgroup of $\overline{G}$ of positive dimension. We start by proving three lemmas where, in fact, the maximality assumption for $G_1$ is not necessary.

Set $\rho$ to be a rational representation of $\overline{G}$ and let $c$ be an upper bound for $\deg_{\rho}(\overline{G_1})$; note that, by \eqref{g0deg}, we also have $|\ov G_1:(\ov G_1)^0|\leq c$.

For each $i=2,\dots, k$, we define $\overline{G_i}=\overline{G_{i-1}}\cap \overline{G_1}^{g_i}$ where $g_i$ is the element of $G$ mentioned above. Thus we have a chain of subgroups
 \begin{equation}\label{e: chain 2}
  \overline{G} > \overline{G_1} \geq \overline{G_2} \geq \cdots \geq \overline{G_k}.
 \end{equation}

\begin{lem}\label{l: stab chain 2}
The subgroups $G_1,\dots, G_k$ in \eqref{e: stab chain} satisfy $G_i=\overline{G_i}^F\cap G$ for each $i=1,\dots, k$.
\end{lem}
\begin{proof}
 We proceed by induction on $i$. The result is true for $i=1$. We assume the result is true for $i$ and prove it for $i+1$. Note that $G_{i+1}=G_i\cap G_1^{g_{i+1}}$ and $\overline{G_{i+1}}=\overline{G_i}\cap \overline{G_1}^{g_{i+1}}$.
 
Let $x\in \overline{G_{i+1}}^F \cap G$. This is equivalent to
\begin{align*}
  & x \in (\overline{G_i} \cap \overline{G_1}^{g_{i+1}})^F \cap G \\
 \Leftrightarrow & x \in (\overline{G_i}^F \cap (\overline{G_1}^{g_{i+1}})^F)\cap G & \\
 \Leftrightarrow & x \in (\overline{G_i}^F \cap G) \cap ((\overline{G_1}^{g_{i+1}})^F)\cap G) & \\
 \Leftrightarrow & x \in (\overline{G_i}^F \cap G) \cap ((\overline{G_1}^{F})^{g_{i+1}})\cap G) & \\
 \Leftrightarrow & x \in (\overline{G_i}^F \cap G) \cap (\overline{G_1}^{F})\cap G)^{g_{i+1}} & \\
 \Leftrightarrow & x\in G_i \cap G_1^{g_{i+1}} = G_{i+1}. & &\qedhere
 \end{align*}
\end{proof}

The lemma implies, in particular, that all of the containments in \eqref{e: chain 2} are proper. Let $d_1=\dim(\overline{G_1})$. Then of course $d_1<d =\dim \overline{G}$. Note that $\overline{G_1}$ is the largest group in the chain \eqref{e: chain 2} of dimension $d_1$.

Now let $k_1,\ldots, k_s$ be the points in the chain \eqref{e: chain 2} where the dimension drops: that is, $k_1=1$, and for each $i\ge 2$, 
$\overline{G_{k_i}}$ is the largest group in the chain such that $\dim \ov G_{k_i} < \dim \ov G_{k_i-1}$. Obviously $s\le d_1+1 \le d$.

\begin{lem}\label{l: deg chain}
We have $\deg_{\rho}\overline{G_{k_i}} \le c^i$.
\end{lem}
\begin{proof}
 We proceed by induction on $i$. For $i=1$, $\overline{G_{k_1}}=\overline{G_1}$ and this has degree at most $c$. We assume the result is true for $i$ and prove it for $i+1$. In particular this means that 
$\overline{G_{k_i}}$ has degree at most $c^i$. Consider the chain
 \[
  \overline{G_{k_i}} > \overline{G_{k_i+1}} > \overline{G_{k_i+2}} > \cdots > \overline{G_{k_{i+1}}}
 \]
Notice that, all but the last listed group have the same dimension, and so have the same identity component; what is more the number of components decreases as we descend the chain from $\overline{G_{k_i}}$ to $\overline{G_{k_{i+1}-1}}$. Thus \eqref{g0deg} implies that
\[
 \deg_{\rho}(\overline{G_{k_{i+1}-1}}) \leq \deg_{\rho}(\overline{G_{k_i}}) \leq c^i.
\]
Now $\overline{G_{k_{i+1}}}$ is the intersection of $\overline{G_{k_{i+1}-1}}$ and a conjugate of $\overline{G_1}$. The former has degree at most $c^i$, and the latter has degree at most $c$. Hence Proposition \ref{t: bezout} implies that $\deg_{\rho}(\overline{G_{k_{i+1}}}) \le c^{i+1}$, as required.
\end{proof}
 

\begin{lem}\label{leng}
The length $k$ of the stabilizer chain $(\ref{e: stab chain})$ satisfies $k \le d + \frac12 d(d+1) \log_2 c$.
\end{lem}

\begin{proof}
 The previous lemma asserts that the degree of $\overline{G_{k_i}}$ is at most $c^i$ and so we also know that $|\overline{G_{k_i}}:(\overline{G_{k_i}})^0| \leq c^i$. Now, for each $i=1,\dots, s$, we know that
 \[
  \overline{G_{k_i}} > \overline{G_{k_i+1}} > \overline{G_{k_i+2}} > \cdots > \overline{G_{k_{i+1}-1}}\geq (\overline{G_{k_i}})^0
 \]
where $\overline{G_{k_i}}^0$ is the identity component of all of the groups in this chain. Since $|\overline{G_{k_i}}:(\overline{G_{k_i}})^0| \leq c^i$,  the length of the chain
\[
  \overline{G_{k_i}} > \overline{G_{k_i+1}} > \overline{G_{k_i+2}} > \cdots > \overline{G_{k_{i+1}-1}}
\]
is at most $\log_2 (c^i) = i\log_2 c$; in particular, for $i=1,\dots, s$, the length of the chain from $ \overline{G_{k_i}}$ to $\overline{G_{k_{i+1}}}$ is at most $i\log_2c+1$. There are two further parts of the chain that we have not considered.

First, at the top of the chain, the containment $G>G_1=G_{k_1}$ adds 1 to the total length. Second, at the bottom of the chain, $\overline{G_{k_{s}}}$ is of dimension $0$ and degree at most $c^{s}$; in other words $\overline{G_{k_{s}}}$ has cardinality at most $c^{s}$ and so there at most $\log_2(c^{s})$ further containments at the end of the chain from 
$\overline{G_{k_{s}}}$ to $\{1\}$.

Our total chain length is, then, at most 
\[ 
1+ \sum_{i=1}^{s-1} (i\log_2c+1) + s\log_2c = s+\frac{1}{2}s(s+1)\log_2c.
\]
Since $s \le d$, the conclusion follows. \end{proof}

\vspace{2mm}

We are ready to complete the proof of Theorem~\ref{t: i} in this case. We reinstate the maximality supposition on $G_1$. We consider the adjoint representation, $ad$, of $\overline{G}$ and we set
\[
c = |W(\ov G)|\cdot 2^{d^2},
\]
For the moment we exclude the exceptional cases $(\ov G,\ov G_1,p) = (C_n,D_n,2)$ or $(C_3,G_2,2)$ in Theorem \ref{t: alg}(1)(b); then, by Theorem~\ref{t: alg}(1), $c$ is an upper bound for $\deg_{ad}(\overline{G_1})$ and also, by \eqref{g0deg}, for $|\ov G_1:(\ov G_1)^0|$.

Recall that $r$ is the rank of $\ov G$, and that $d = \dim \ov G$, so that $d \le 4r^2$. Also $c = |W(\ov G)|\cdot 2^{d^2} \le 2^{r^2+d^2} \le 2^{r^2+16r^4}$. Hence Lemma \ref{leng} gives 
\[
k \le 4r^2 + \frac{1}{2}(4r^2)(4r^2+1)(r^2+16r^4).
\]
The right hand side is at most $Cr^8$ with $C = 174$, as required for Theorem \ref{t: i}.

It remains to deal with the excluded cases $(\ov G,\ov G_1,p) = (C_n,D_n,2)$ or $(C_3,G_2,2)$. In the former case \cite[Lemma~6.11]{glodas} implies that $\Irred(G,\Omega)\leq 2r+1$ and the conclusion holds. In the latter case the action of $G = C_3(q)$ on $\Omega = (C_3(q):G_2(q))$ is contained in $(D_4(q):(D_4(q):B_3(q))$, since there is a factorization $D_4(q) = AB$, where $A\cong B \cong B_3(q)$ and $A\cap B \cong G_2(q)$ (see \cite[p.105]{LPS}). For this action of $X:=D_4(q)$, we have $\Irred(X,\Omega) \le 15$ by \cite[3.1]{kelsey2021relational}. Hence $\Irred(G,\Omega) \le 15$. 

This completes the proof of Theorem \ref{t: i}.

\subsection{Case~2 of Theorem \ref{t: alg}}

In this case we have $G_1 = G(q_0)$, a subgroup of the same type as $G$ (possibly twisted) over a subfield $\bF_{q_0}$ of $\bF_q$. Writing $G = (\overline G^F)'$ as before, there is a Frobenius endomorphism $F_0$ of $\overline G$ such that $G_1 = \overline G^{F_0} \cap G$, where $F_0^r = F$ for some integer $r\ge 2$. 

\begin{lem}\label{q0int} For $x \in G$ we have 
\[
G_1\cap G_1^x = C_{G_1}(x^{-1}x^{F_0}) = (C_{\ov G}(x^{-1}x^{F_0}))^{F_0}.
\]
\end{lem}

Note that the group $C_{\ov G}(x^{-1}x^{F_0})$ may not be $F_0$-stable.

\begin{proof} We have 
\[
\begin{array}{ll}
g \in G_1\cap G_1^x & \Leftrightarrow g,g^{x^{-1}} \in G_1 \\
                               & \Leftrightarrow g^{F_0} = g \hbox{ and } (xgx^{-1})^{F_0} = xgx^{-1} \\
                               &  \Leftrightarrow g^{F_0} = g \hbox{ and } x^{F_0}gx^{-F_0}=xsx^{-1} \\
                               & \Leftrightarrow g \in C_{G_1}(x^{-1}x^{F_0}).
\end{array}
\]
\end{proof}

Recall that we have a stabilizer chain $ G>G_1>G_2 > \cdots > G_k=1$, where $G_i=G_{i-1}\cap G_1^{g_i}$ for each $i$, and $g_i \in G$. Define 
\[
\overline G_1 = \overline G,\;\;\overline G_2 = C_{\ov G}(g_2^{-1}g_2^{F_0}),
\]
and for $2 \le j \le k$,
\[
\overline G_j = \bigcap_{i=2}^j C_{\ov G}(g_i^{-1}g_i^{F_0}).
\]
Then by Lemma \ref{q0int}, we have $G_j = \overline G_j^{F_0}$ for $1\le j \le k$, and so 
\[
\overline G = \overline G_1 > \overline G_2 > \cdots > \overline G_k.
\]
Given $x\in \overline G$, we of course have $C_{\ov G}(x) = \{g \in \overline G : gx = xg\}$, so this centralizer consists of solutions of a system of linear equations in the entries of $g$, and hence $\deg_{ad}C_{\ov G}(x) \le \deg_{ad}\ov G$.
 Now we can bound the length $k$ of the chain exactly as in Case 1, and the proof is complete.

\subsection{Case~3 of Theorem \ref{t: alg}}
This case is a triviality: clearly if $|G_1|\leq 2^{d^2}$, then a stabilizer chain has length at most $d^2$. This observation completes the proof of Theorem~\ref{t: i}. $\;\;\;\Box$

\vspace{2mm}

\begin{example}\label{complem} {\rm Here is an example that shows there is no general complementary 
 \emph{lower bound} to go with the upper bound given in Theorem~\ref{t: i}. Let $G=\SL_r(2)$ acting on $\Omega$, the set of cosets of $H$ where $H$ is the normalizer of a Singer cycle, with $r$ an odd prime. Then $H\cong C_{2^r-1}\rtimes C_r$ and $H$ is maximal in $G$ for $r\ge 13$ (see \cite[Table 3.5A]{KL}). Since distinct conjugates of the Singer cycle $C_{2^r-1}$  intersect trivially, it follows that for this action we have $\Irred(G,\Omega)\leq 3$. In particular, $\Irred(G,\Omega)$ does not necessarily grow as the rank increases, even when $G$ is simple and the action is primitive.}
\end{example}

\begin{remark}\label{rem}
{\rm It is possible to improve the polynomial bound of Theorem~\ref{t: i} in particular cases. For example, consider parabolic actions of $G =\PSL_{n}(q)$ -- i.e. transitive actions for  which the stabilizer $G_1$ is a parabolic subgroup. Set $\ov G=\SL_{n}(\ov \Fq)$ and let $\rho$ be the usual $n$-dimensional rational representation. In this situation, parabolic subgroups $\ov{G_1}$ satisfy $\deg_\rho(\ov{G_1})\leq n$ and so Lemma \ref{leng} gives  $\Irred(G,\Omega)\leq n^4\log_2n$.}
\end{remark}

\section{Almost simple groups: proof of Corollary \ref{c: ii}} \label{corpf}

Let $G$ be an almost simple group, with socle $S = G(q)$, a simple group of Lie type of rank $r$ over $\Fq$, where $q=p^f$ ($p$ prime). Let $G$ acts primitively on a set $\Omega$, with point-stabilizer $G_1$, and let $M_1=G_1\cap S$. Note that  $G=G_1S$, and so $G_1/M_1 \cong G/S$, a subgroup of ${\rm Out}(S)$.

Now let $G>G_1>G_2>\cdots > G_k=\{1\}$ be a stabilizer chain, where $G_i = G_{\a_1\cdots \a_i}$ for $1\le i\le k$.
Define $M_i=G_i\cap S$. We obtain two chains,
\begin{align*}
 S> M_1\geq M_2\geq &\cdots \geq M_k=\{1\}; \\
 G/S = G_1/M_1 \geq G_2/M_2 \geq & \cdots \geq G_k/M_k=\{1\}.
\end{align*}
Observe that, for $i=1,\dots, k-1$, if $M_i=M_{i+1}$, then $G_i/M_i > G_{i+1}/M_{i+1}$. We know that a proper subgroup chain in $G/S$ has length at most $\log_2(|{\rm Out}(S)|) \le  \log_2(6r)+\pi(f)$. Now define
\[
 I=\{i \mid 1\leq i \leq k-1 \textrm{ and } M_i > M_{i+1}\}
\]
and write $I=\{i_1,\dots, i_{\ell-1}\}$ where $i_j < i_{j+1}$ for $j=1,\dots, \ell-2$. Setting $i_\ell=k$ we have, firstly, that
\begin{equation}
\ell \geq k - \log_2(6r)-\pi(f)\label{comb}
\end{equation}
and, secondly, that
\begin{equation}
 S>M_{i_1}>M_{i_2} > \cdots > M_{i_\ell}=\{1\}.\label{ch}
\end{equation}
Note that $i_1=1$, and \eqref{ch} is the stabilizer chain $S>S_{\a_1} > S_{\a_1\a_{i_2}} > \cdots $
for the action of $S$ on $\Omega$.

Now Theorem \ref{t: alg} tells us that $S_{\a_1}$ satisfies (1), (2) or (3) of the conclusion of that theorem. Hence, arguing 
exactly as in the proof of Theorem~\ref{t: i} we obtain that $\ell\leq 174r^8$. Combining this bound with \eqref{comb} yields $k\leq 174r^8+\log_2(6r) + \pi(f)$, which is less than $177r^8 + \pi(f)$. This completes the proof of Corollary \ref{c: ii}.  $\;\;\Box$

\begin{example}\label{pif} {\rm Here is an example that shows that the term $\pi(f)$ in the upper bound in Corollary \ref{c: ii} cannot be avoided. 

Let $G=\PGammaL_2(q)$ with $q=p^f$, and consider the action of $G$ on the set of $1$-subspaces of $V=(\mathbb{F}_q)^2$. We claim that $\Irred(G,\Omega)=3+\pi(f)$. To see this, write the prime factorization of $f$ as $f=r_1r_2\cdots r_{\ell}$ where $\ell=\pi(f)$, write $\{e_1, e_2\}$ for the natural basis of $V$ over $\Fq$, and consider the stabilizer chain obtained by successively stabilizing the following $1$-spaces (in order):
\[
 \langle e_1\rangle, \,
 \langle e_2\rangle, \,
 \langle e_1+e_2\rangle, \,
 \langle e_1+\zeta_1 e_2\rangle, \, \langle e_1+\zeta_2e_2\rangle, \dots ,
 \langle e_1+\zeta_\ell e_2 \rangle,  
\]
where, for $i=1,2,\dots, \ell$, $\zeta_i$ is a primitive element of $\mathbb{F}_{p^{r_1r_2\cdots r_i}}$. This stabilizer chain establishes that $\Irred(G,\Omega)\geq 3+\pi(f)$; on the other hand the 3-transitivity of the action of $G$ implies that the stabilizer of any 3 distinct points is isomorphic to $C_f$ and this implies that $\Irred(G,\Omega)\leq 3+\pi(f)$. }
\end{example}

It seems possible, however, that one could do better for $\Base(G,\Omega)$ and/or $\Height(G,\Omega)$. In the proof of Lemma~\ref{l: bb} below we shall show that there exists a primitive action of $G=\PGammaL_2(q)$ for which $\Base(G,\Omega)\geq\pi_d(f)$, where $\pi_d(f)$ is the number of \emph{distinct} primes dividing the integer $f$.

\begin{conj}
There exists a function $g:\mathbb{Z}^+\to \mathbb{Z}^+$ such that if $G$ is an almost simple group of Lie type of rank $r$ over a field of order $p^f$ acting primitively on a set $\Omega$, then 
\[
 \Base(G,\Omega) \leq \Height(G,\Omega) < g(r)+\pi_d(f) \textrm{ and } 
\]
where $\pi_d(f)$ is the number of distinct primes dividing the integer $f$.
\end{conj}


\section{Theorem~\ref{t: i} and the Cameron--Kantor conjecture}\label{s: irred ck}

The Cameron--Kantor Conjecture (now a theorem due to Liebeck and Shalev \cite{lieshal2}) asserts the following:
\begin{quotation}
{\it There exists a constant $c>0$ such that if $G$ is an almost simple primitive non-standard permutation group on a set $\Omega$, then $\base(G,\Omega)\leq c$.} 
\end{quotation}
(A \emph{standard action} of an almost simple group $G$ with socle $S$ is a transitive action where either $S=A_n$ and the action is on subsets or uniform partitions of $\{1,\dots, n\}$, or $G$ is classical and the action is a subspace action; see \cite{burness} for more detail.) This statement is now known to be true with $c=7$ , by\cite{burness, burgs, burls, burness, burow}. 

Colva Roney-Dougal asked us whether a statement like the Cameron--Kantor conjecture might be true for any of the statistics $\Base(G,\Omega)$, $\Height(G,\Omega)$ or $\Irred(G,\Omega)$ and Theorem~\ref{t: i} was our answer to this question. One naturally wonders, though, whether it is possible to do better -- to investigate this, given \eqref{eq: in1}, the first question one should ask is whether a stronger statement can be proved for $\Base(G,\Omega)$ (since any such statement for $\Height(G,\Omega)$ or $\Irred(G,\Omega)$ is necessarily true for $\Base(G,\Omega)$). To investigate this we need to clarify some things.

\vspace{2mm}
{\bf Primitivity and transitivity.} Suppose that $G$ is a transitive permutation group on $\Omega$ and identify $\Omega$ with $(G:H)$ where $H$ is the stabilizer of a point. Now let $F\leq H$ and let $\Gamma = (G:F)$. Then it is true that $\base(G,\Gamma)\leq \base(G,\Omega)$ hence, in particular, the Cameron--Kantor Conjecture gives information about all transitive almost simple permutation groups $G$ for which a point-stabilizer is a subgroup of a maximal subgroup that is a point stabilizer for a non--standard primitive action.

Things are more complicated for us because it is not necessarily true that $\Base(G,\Gamma)\leq \Base(G,\Omega)$, that $\Height(G,\Gamma)\leq \Height(G,\Omega)$ or that $\Irred(G,\Gamma)\leq \Irred(G,\Omega)$; the examples below demonstrate this. Hence in investigating how to extend the statement of the Cameron--Kantor Conjecture we need to distinguish between statements involving primitive groups and those involving transitive groups.

\vspace{2mm}
{\bf Rank-dependent constant versus absolute constant}. Our investigations will focus on almost simple groups with socle a group of Lie type. Our first example will establish that it is not possible to  give an absolute upper bound for $\Base(G,\Omega)$, even for non-standard actions. In light of this it is worth clarifying what the Cameron--Kantor Conjecture implies with regard to a rank-dependent upper bound. It is the following:
\begin{quotation}
{\it For every positive integer $r$ there exists a constant $c>0$ such that if $G$ is an almost simple primitive permutation group on a set $\Omega$, with socle a group of Lie type of rank at most $r$, then $\base(G,\Omega)\leq c$.} 
\end{quotation}
The point we are making here is that, if we allow our upper bound to be rank-dependent, then we do not need to distinguish between standard and non-standard actions -- it is easy enough to establish that the standard actions also satisfy the given statement. (For the $\mathcal{C}_8$ standard actions of $\Sp_{2m}(q)$ this follows from \cite[Lemma~6.11]{glodas}; for the $\mathcal{C}_1$ standard actions of the classical groups this follows from \cite[Theorem~3.1]{kelsey2021relational}.)

Note, finally, that we have not considered the question of Cameron--Kantor like statements for irredundant bases of primitive actions of the alternating groups.

\subsection{Simple, primitive, absolute upper bound}

In this subsection we show that the following possible extension of the Cameron--Kantor Conjecture is false:
\begin{quotation}
{\it There exists a constant $c>0$ such that if $G$ is a simple primitive non-standard permutation group on a set $\Omega$, then $\Base(G,\Omega)\leq c$.}
\end{quotation}
The key point here is that an upper bound on $\Base(G,\Omega)$ in this setting must depend on $r$.

\begin{lem}
For every $n\ge 13$, $q\ge 5$, there exists a non-standard primitive action $(\PSL_n(q),\Omega)$ such that 
$\Base(\PSL_n(q),\Omega)\ge n-1$.
\end{lem}
\begin{proof}
We consider the action of $G=\SL_n(q)$ acting on the cosets of a $\mathcal{C}_2$-maximal subgroup that is the normalizer of a split torus. For $q\ge 5$, $n\ge 13$ this induces a primitive non-standard action of $\PSL_n(q)$ (see \cite[Table 3.5A]{KL}); furthermore this action of $G$ is equivalent to the action of $G$ on decompositions of $V=(\mathbb{F}_q)^n$ as a direct sum of $n$ $1$-dimensional subspaces.

Let $\{e_1,\dots, e_n\}$ be a basis of $V$ over $\Fq$. For $i=1,\dots, n-1$, we define a decomposition $\mathcal{D}_i$ of $V$ as follows:
\[
 \mathcal{D}_i = \langle e_1\rangle \oplus \langle e_2\rangle \oplus \cdots \oplus \langle e_{i-1}\rangle \oplus \langle e_i+e_{i+1}\rangle \oplus \langle e_{i+1}\rangle \oplus \langle e_{i+2}\rangle \oplus \cdots \oplus \langle e_n\rangle.
\]

Suppose, first, that $g\in G$ fixes $\mathcal{D}_1,\dots, \mathcal{D}_{n-1}$. This implies that $g$ fixes the space $\langle e_n\rangle$ (since it is the only $1$-space appearing in all $n-1$ decompositions; similarly, for $j=1,\dots, n-1$, $g$ fixes the space $\langle e_j\rangle$ (since it is the only $1$-space appearing in all $n-1$ decompositions except for $\mathcal{D}_j$). Thus, for $j=1,\dots, n$, there exists $\lambda_j\in\Fq$ such that $e_j^g=\lambda_je_j$. But now, for $j=1,\dots, n-1$, the space $\langle e_j+e_{j+1}\rangle$ occurs in decomposition $\mathcal{D}_j$ and no others, hence this 1-space too is fixed by $g$. This implies, finally, that, for $j=1,\dots, n-1$, $\lambda_{j}=\lambda_{j+1}$ and so $g$ acts as a scalar. In particular, the set $\{\mathcal{D}_1,\dots, \mathcal{D}_{n-1}\}$ is a base for the induced action of $\PSL_n(q)$.

On the other hand, for $j\in 1,\dots, n-1$, define $\Lambda_j=\{\mathcal{D}_1,\dots, \mathcal{D}_{j-1}, \mathcal{D}_{j+1},\dots,\mathcal{D}_{n-1}\}$ and set $g_j$ to be an element of $G$ that swaps $\langle e_j\rangle$ and $\langle e_n\rangle$ while fixing $\langle e_i\rangle$ for $i=1,\dots, j-1, j+1,\dots, n-1$. It is straightforward to check that $g$ fixes all of the decompositions in $\Lambda_j$. We conclude that $\Lambda$ is a minimal base for this action of size $n-1$.
 \end{proof}

In light of this lemma our remaining investigations will focus on almost simple groups where the socle is a group of Lie type of bounded rank.

\subsection{Simple, transitive, rank-dependent upper bound}\label{s: stb}

In this subsection we show that the following possible extension of the Cameron--Kantor Conjecture is false:
\begin{quotation}
{\it For every positive integer $r$ there exists a constant $c>0$ such that if $G$ is a simple transitive permutation group on a set $\Omega$, with socle a group of Lie type of rank at most $r$, then $\Base(G,\Omega)\leq c$.}
\end{quotation}

The next lemma does the job:

\begin{lem}
 For every integer $c>1$, there exists a transitive action $(\SL_2(2^c),\Omega)$, such that $\Base(\SL_2(2^c),\Omega) \ge c$.
\end{lem}
\begin{proof}
Let $q=2^c$, let $G=\SL_2(q)$, let $U$ be a Sylow $2$-subgroup of $G$, let $H$ be an index $2$ subgroup of $U$ and let $\Omega$ be the set of right cosets of $H$ in $G$. Since $H=2^{c-1}$ it is clear that $\Base(G,\Omega)\leq \Irred(G,\Omega)\leq c$. We claim that, in fact, $\Base(G,\Omega)=c$.

To show this, let $B=N_G(U)$ and let $\Delta$ be the set of right cosets of $H$ in $B$. Since $\Base(B,\Delta)\leq \Base(G,\Omega)$ it is sufficient to show that $\Base(B,\Delta)\geq c$.

Consider $U$ as a $c$-dimensional vector space over $\mathbb{F}_2$. The action o $B$ on $\Delta$ is isomorphic to the action of $B$ on the set of all \emph{affine hyperplanes} -- these are the usual linear hyperplanes as well as their translates. Since we are working over $\mathbb{F}_2$, each hyperplane has 2 cosets (itself and one other) thus $|\Delta|=2(q-1)$.

Observe that if $H_1$ is a linear hyperplane, then the stabilizer of $H_1$ in $B$ is $H_1$ itself (in particular, $H_1$ is a conjugate of $H$). Let $e_1,\dots, e_c$ be the usual vectors in the natural basis of $U$ (so $e_i$ has $0$'s in all places except the $i$-th where the entry is $1$). For $i=1,\dots, c$, define
\[
 H_i:=\langle e_1,\dots, e_{i-1}, e_{i+1},\dots, e_c\rangle.
\]
Then $H_1,\dots, H_c$ are linear hyperplanes in $U$ hence are elements of $\Delta$ and conjugates of $H$. For $i=j,\dots, c$, define $\Lambda_j =\{H_1,\dots, H_{j-1}, H_{j+1},\dots, H_c\}$ and observe that $B_{(\Lambda_j)}=\langle e_j\rangle$. Thus $\Lambda=\{H_1,\dots, H_c\}$ is a minimal base of size $c$.
\end{proof}

\subsection{Almost simple, primitive, rank-dependent upper bound}

In this subsection we show that the following possible extension of the Cameron--Kantor Conjecture is false:
\begin{quotation}
{\it For every positive integer $r$ there exists a constant $c>0$ such that if $G$ is an almost simple primitive permutation group on a set $\Omega$, with socle a group of Lie type of rank at most $r$, then $\Base(G,\Omega)\leq c$.}
\end{quotation}

The next lemma does the job:

\begin{lem}\label{l: bb}
For every $c>0$, there exists a non-standard primitive action $(\PGammaL_2(q),\Omega)$, for some $q$, such that $\Base(\PGammaL_2(q),\Omega)>c$.
\end{lem}
\begin{proof}
Let $G=\mathrm{\Gamma L}_2(q)$ and consider the action on cosets of the normalizer of a split torus.  For $q>11$ this induces a primitive non-standard action of $\PGammaL_2(q)$; furthermore this action of $G$ is equivalent to the action of $G$ on decompositions of $V=(\mathbb{F}_q)^2$ as a direct sum of two $1$-dimensional subspaces. Let $q=p^d$ and assume that $d=f_1\cdots f_k$ where $k\geq 3$ and $f_1,\dots, f_k$ are distinct primes.

Let $\{e_1, e_2\}$ be the natural basis for $V$ over $\Fq$: $e_1=(1\; 0)$ and $e_2=(0\;1)$. We define decompositions $\mathcal{D}_i$ for $i=1,\dots, k$ as follows:
\[
 \mathcal{D}_i: \left\langle e_1\right\rangle \oplus \left\langle e_1+\zeta_i e_2\right\rangle
\]
where $\zeta_i$ is a primitive element in $\mathbb{F}_{p^{f_i}}$. To see that 
$\mathcal{D}_1,\ldots ,\mathcal{D}_k$  form an independent set we consider the action $F=\langle \sigma\rangle<G$ where $\sigma$ is the field automorphism that acts on vectors by raising each entry to the $p$-th power.

For $j\in 1,\dots, k$, define $\Lambda_j=\{\mathcal{D}_1,\dots, \mathcal{D}_{j-1}, \mathcal{D}_{j+1},\dots,\mathcal{D}_{k}\}$. The pointwise-stabilizer of $\Lambda_j$ in $F$ is $\langle \sigma^{d/f_j}\rangle$ and so the pointwise-stabilizers of $\Lambda_j$ are distinct for $j=1,\dots, k$; in particular we obtain that $\Lambda=\{\mathcal{D}_1,\dots, \mathcal{D}_k\}$ is an independent set of size $k$.

We claim that, in fact, $\Lambda$ is a minimal base. To see this, we must prove that the pointwise-stabilizer of $\Lambda$ is trivial. Let $g\in G_{(\Lambda)}$ and write $g=\sigma^r x$ where $r$ is some positive integer and $x\in \GL_2(q)$; without loss of generality we can assume that $r$ divides $d$. It is clear that $\langle e_1\rangle^g=\langle e_1\rangle$ thus there exists $\lambda_0\in \Fq$ such that 
\[
 \lambda_0 e_1 = e_1^g = e_1^{\sigma^r x} = e_1^x.
\]

Similarly, for $i=1,\dots, k$, there exists $\lambda_i\in\Fq$ such that
\begin{align*}
 \lambda_i (e_1+\zeta_i e_2) &= (e_1+\zeta_i e_2)^g \\
 &= e_1^g + (\zeta_i e_2)^g \\
 &= e_1^x + \zeta_i^{\sigma^r} e_2^x \\
 &= \lambda_0 e_1 + \zeta_i^{p^r} e_2^x.
 \end{align*}
Rearranging we obtain that
\[
 e_2^x=\lambda_i \zeta_i^{1-p^r} e_2 + \zeta_i^{-p^r}(\lambda_i-\lambda_0)e_1.
\]
We conclude that, for distinct $i,j\in \{1,\dots, k\}$ we have
\[
 \lambda_i\zeta_i^{1-p^r} = \lambda_j\zeta_j^{1-p^r} \textrm{ and } \zeta_i^{-p^r}(\lambda_i-\lambda_0) = \zeta_j^{-p^r}(\lambda_j-\lambda_0).
\]
The latter equation yields that $\lambda_i=\left(\frac{\zeta_i}{\zeta_j}\right)^{p^r}\lambda_j+\left(1-\left(\frac{\zeta_i}{\zeta_j}\right)^{p^r}\right)\lambda_0$ while the former yields that $\lambda_i = \frac{\zeta_i^{p^r-1}}{\zeta_j^{p^r-1}}\lambda_j$. Combining these two identities and rearranging yields
\[
 \left(\frac{\zeta_j/\zeta_i - 1}{(\zeta_j/\zeta_i)^{p^r}-1}\right)\lambda_j=\lambda_0.
\]
If we fix $j$ and choose $\ell,m\in\{1,\dots, k\}$ such that $j$, $\ell$ and $m$ are all distinct, then we obtain that
\[
 \frac{\zeta_j/\zeta_\ell - 1}{(\zeta_j/\zeta_\ell)^{p^r}-1}=\frac{\zeta_j/\zeta_m - 1}{(\zeta_j/\zeta_m)^{p^r}-1}
\]
and, rearranging, we have
\[
 \left(\frac{\zeta_j/\zeta_\ell-1}{\zeta_j/\zeta_m-1}\right)^{p^r-1}=1.
\]
We claim that the smallest field containing the quantity in parenthesis is either $\mathbb{F}_{p^{f_jf_\ell f_m}}$ or $\mathbb{F}_{p^{f_\ell f_m}}$. To see this, denote this quantity $\eta$ and suppose that $\eta$ is contained in $\mathbb{F}_{p^{f_jf_\ell}}$. Rearranging we obtain
\[
 \zeta_m=\frac{\zeta_j\zeta_\ell\eta}{\zeta_j-\zeta_\ell+\zeta_\ell\eta} \in \mathbb{F}_{p^{f_jf_\ell}},
\]
a contradiction. A similar argument allows us to conclude that this quantity is not contained in $\mathbb{F}_{p^{f_jf_m}}$ and the claim follows.

We obtain that $r$ is divisible by both $f_\ell$ and $f_m$. Repeating this argument we obtain that $r$ is divisible by all primes $f_1,\dots, f_k$ and thus $g=x$. But this implies that $\lambda_i=\lambda_j=\lambda_0$ for all $i,j=1,\dots, k$ and $g$ is a scalar, as required.

We conclude that $\Lambda$ is a minimal base for this action. Since $|\Lambda|=k$, we need only choose $k>c$ to obtain that $\Base(G,\Omega)\geq k>c$ as required.
\end{proof}

\subsection{Simple, primitive, rank-dependent upper bound}

In light of the examples given in the preceding sections, this is the only setting where a direct extension of Cameron--Kantor is possible. As mentioned above, if we allow our upper bound to be rank-dependent, then we can ignore the distinction between standard and non-standard actions, hence the statement we end up with has the form of Theorem~\ref{t: i}.




\end{document}